\documentclass[11pt,letterpaper]{article}

\usepackage{graphicx}
\usepackage{setspace}
\usepackage{amssymb,amsmath}
\usepackage{calc}
\usepackage{verbatim}
\usepackage{epsfig}
\usepackage{graphicx}
\usepackage{tikz, tikz-cd,subcaption}
\usetikzlibrary{arrows}
\usepackage{graphics}
\usepackage{enumitem}

\newtheorem{theorem}{Theorem}[section]

\newtheorem{prob}{Problem}
\newenvironment{proof}           {\noindent{\bf Proof.} }%
                                {\null\hfill$\Box$\par\medskip}

\newcommand{\lnk}{\text{link}}
\newcommand{\del}{\text{del}}
\newcommand{\Rel}[1]{\mbox{\mbox{Rel}$(#1,q)$}}

\newcommand{\CRel}[1]{\mbox{\mbox{CSRel}$(#1,p)$}}
\newcommand{\reff}[1]{(\ref{#1})}
\newcommand{\M}{\mathcal{M}}
\newcommand{\mgen}[1]{\mbox{\mbox{mgen}$_\mathcal{M}(#1)$}}

\begin{document}

\title{On the Reliability Roots of Simplicial Complexes and Matroids}


\author{J.I. Brown\footnote{Corresponding author. {\tt jason.brown@dal.ca}} and C.D.C. DeGagn\'e\\
Department of Mathematics and Statistics, Dalhousie University,\\ Halifax, Canada B3H 4R2}

\date{}

\maketitle

\begin{abstract}
Assume that the vertices of a graph $G$ are always operational, but the edges of $G$ fail independently with probability $q \in[0,1]$. The \emph{all-terminal reliability} of $G$  is the probability that the resulting subgraph is connected. The all-terminal reliability can be formulated into a polynomial in $q$, and it was conjectured \cite{BC1} that all the roots of (nonzero) reliability polynomials fall inside the closed unit disk.  It has since been shown that there exist some connected graphs which have their reliability roots outside the closed unit disk, but these examples seem to be few and far between, and the roots are only barely outside the disk.
In this paper we generalize the notion of reliability to simplicial complexes and matroids and investigate when, for small simplicial complexes and matroids, the roots fall inside the closed unit disk.
\end{abstract}

\section{Introduction}

A well known model of network robustness is the \emph{(all terminal) reliability} of a graph, in which the vertices are always operational, but each edge fails independently with probability $q \in [0,1]$. The reliability of the undirected, connected graph $G$, $\Rel{G}$, is the probability that the operational edges form a connected spanning subgraph, that is, that the operational edges contain a spanning tree. It is easy to see that the reliability of a graph $G$ is always a polynomial in $q$ (and in $p = 1-q$) and is not identically $0$ if and only if $G$ is connected.  As such, one can express the reliability of a graph $G$ of order $n$ and size $m$ (that is, with $n$ vertices and $m$ edges) in a variety of useful forms, by expanding the polynomial in terms of different bases (see, for example, \cite{C1}):
	
\begin{eqnarray}
\Rel{G} & = & \sum_{i = 0}^{m-n+1} F_iq^i(1-q)^{m-i}~~~~~(\mbox{F-Form}) \label{Fform}\\
             & = & (1-q)^{n-1}\sum_{i=0}^{m-n+1}H_iq^i~~~~(\mbox{H-Form}).  \label{Hform}
\end{eqnarray}

The interpretation of the coefficients of the F-form are quite straightforward. A connected spanning subgraph with edge set $E^\prime \subseteq E$ arises with probability $q^{|E \setminus E^\prime|}(1 - q)^{|E^\prime|}$.  Therefore, $F_i$ is the number of sets of $i$ edges whose removal leaves the graph connected. The interpretation of the coefficients of the H-form are more subtle, but we shall to return to them shortly.
	
Most of the early work on reliability concerned algorithmic calculation (see \cite{C1}), and then, when the problem was found to be intractable, emphasis was placed on efficient methods of bounding the function.  However, a natural mathematical question is to ask where the roots of the reliability polynomial lie in the complex plane (the location of roots of polynomials have direct implications for the relationship between the coefficients -- for example, Newton's well known theorem \cite{comtet} shows that a polynomial with positive coefficients having all real roots implies that the coefficients are unimodal).  In 1992 it was conjectured  \cite{BC1} that all the roots of $\Rel{G}$ of a connected graph $G$ (as we shall assume from now on about a graph $G$) lie in the unit disk (closed and centered at the origin, as we shall always assume when talking about the unit disk). See Figure~\ref{relrootsorder7} for a plot of reliability roots of small simple graphs. In support of the conjecture, it was shown in \cite{BC1} that 
\begin{itemize}
\item the real reliability roots of graphs are always in $[-1,0) \cup \{1\}$ (and hence in the unit disk),
\item every graph had a subdivision for which the roots lay in the unit disk, and
\item the closure of the (complex) reliability roots contains the unit disk.
\end{itemize} 
For many years the conjecture was believed to be true, and proved for some classes of graphs such as series-parallel graphs \cite{W1}.  However, in 2004 Royle and Sokal \cite{RS1} constructed graphs whose reliability polynomials $\Rel{G}$ had a root outside of the closed unit disk.  The smallest such example started with the complete graph $K_4$ and then replaced one pair of opposite edges by a bundle of six edges (see Figure~\ref{fig:RSGraph}). This graph has a root of its reliability polynomial whose modulus is approximately $1.0017$, just barely outside of the disk.  There have been subsequent attempts to find roots with larger modulus, with \cite{BM1} pushing the roots out as far from the origin as $1.113$, but still the examples seem few and far between, and still bounded. 

\begin{figure}[ht]
\centering
\includegraphics[width=3.0in]{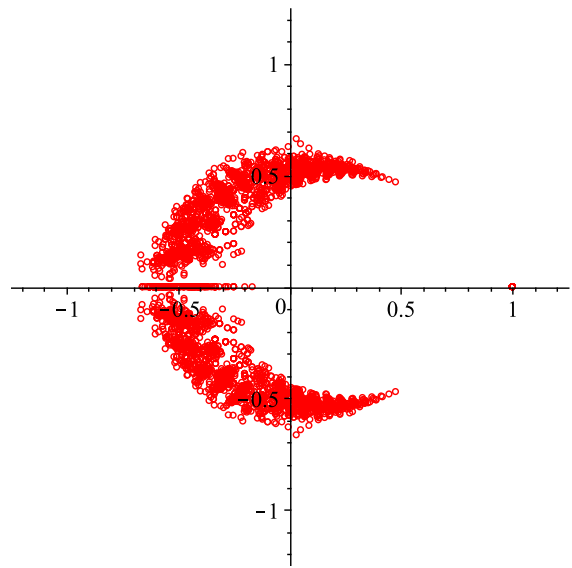}
\caption{The roots of reliability polynomials of all simple graphs of order $7$}.
\label{relrootsorder7}
\end{figure}

	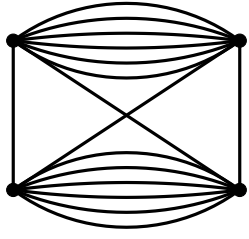
\begin{figure}[ht]
		\begin{center}
			\begin{tikzpicture}[line cap=round,line join=round,>=triangle 45,x=1.003719132544759cm,y=0.99301850213383cm]
\clip(-2.474580852224948,1.0377059052999644) rectangle (2.506892388524535,5.065828232733614);
\draw [shift={(0.,-7.2)},line width=1.2pt]  plot[domain=1.4376599992432495:1.7039326543465438,variable=\t]({1.*11.3*cos(\t r)+0.*11.3*sin(\t r)},{0.*11.3*cos(\t r)+1.*11.3*sin(\t r)});
\draw [shift={(0.,0.4)},line width=1.2pt]  plot[domain=1.1760052070951352:1.965587446494658,variable=\t]({1.*3.9*cos(\t r)+0.*3.9*sin(\t r)},{0.*3.9*cos(\t r)+1.*3.9*sin(\t r)});
\draw [shift={(0.,2.)},line width=1.2pt]  plot[domain=0.9272952180016122:2.214297435588181,variable=\t]({1.*2.5*cos(\t r)+0.*2.5*sin(\t r)},{0.*2.5*cos(\t r)+1.*2.5*sin(\t r)});
\draw [shift={(0.,15.346237600463352)},line width=1.2pt]  plot[domain=4.580948777615711:4.843829183153669,variable=\t]({1.*11.444959925057333*cos(\t r)+0.*11.444959925057333*sin(\t r)},{0.*11.444959925057333*cos(\t r)+1.*11.444959925057333*sin(\t r)});
\draw [shift={(0.,7.6)},line width=1.2pt]  plot[domain=4.317597860684928:5.107180100084451,variable=\t]({1.*3.9*cos(\t r)+0.*3.9*sin(\t r)},{0.*3.9*cos(\t r)+1.*3.9*sin(\t r)});
\draw [shift={(0.,6.)},line width=1.2pt]  plot[domain=4.068887871591405:5.355890089177974,variable=\t]({1.*2.5*cos(\t r)+0.*2.5*sin(\t r)},{0.*2.5*cos(\t r)+1.*2.5*sin(\t r)});
\draw [shift={(0.,-9.2)},line width=1.2pt]  plot[domain=1.4376599992432488:1.7039326543465445,variable=\t]({1.*11.3*cos(\t r)+0.*11.3*sin(\t r)},{0.*11.3*cos(\t r)+1.*11.3*sin(\t r)});
\draw [shift={(0.,-1.6)},line width=1.2pt]  plot[domain=1.1760052070951352:1.965587446494658,variable=\t]({1.*3.9*cos(\t r)+0.*3.9*sin(\t r)},{0.*3.9*cos(\t r)+1.*3.9*sin(\t r)});
\draw [shift={(0.,0.)},line width=1.2pt]  plot[domain=0.9272952180016122:2.214297435588181,variable=\t]({1.*2.5*cos(\t r)+0.*2.5*sin(\t r)},{0.*2.5*cos(\t r)+1.*2.5*sin(\t r)});
\draw [shift={(0.,13.2)},line width=1.2pt]  plot[domain=4.579252652833041:4.845525307936337,variable=\t]({1.*11.3*cos(\t r)+0.*11.3*sin(\t r)},{0.*11.3*cos(\t r)+1.*11.3*sin(\t r)});
\draw [shift={(0.,5.6)},line width=1.2pt]  plot[domain=4.317597860684928:5.107180100084451,variable=\t]({1.*3.9*cos(\t r)+0.*3.9*sin(\t r)},{0.*3.9*cos(\t r)+1.*3.9*sin(\t r)});
\draw [shift={(0.,4.)},line width=1.2pt]  plot[domain=4.068887871591405:5.355890089177974,variable=\t]({1.*2.5*cos(\t r)+0.*2.5*sin(\t r)},{0.*2.5*cos(\t r)+1.*2.5*sin(\t r)});
\draw [line width=1.2pt] (-1.5,4.)-- (-1.5,2.);
\draw [line width=1.2pt] (1.5,4.)-- (1.5,2.);
\draw [line width=1.2pt] (1.5,4.)-- (-1.5,2.);
\draw [line width=1.2pt] (-1.5,4.)-- (1.5,2.);
\begin{scriptsize}
\draw [fill=black] (-1.5,4.) circle (2.5pt);
\draw [fill=black] (1.5,4.) circle (2.5pt);
\draw [fill=black] (-1.5,2.) circle (2.5pt);
\draw [fill=black] (1.5,2.) circle (2.5pt);
\end{scriptsize}
\end{tikzpicture}
\caption{Royle - Sokal Graph}
\label{fig:RSGraph}	
\end{center}
\end{figure}

\section{Generalizing Reliability to Complexes}

In order to better explore reliability roots in relation to the unit disk, we extend the notion of reliability to more abstract structures (\cite{W1} extends to matroids, and \cite{browncolsetsystem} to more general systems). We start not with a graph but with a  \emph{(simplicial) complex} $\mathcal{C}$ of order $m$, which is a collection of subsets (called \emph{faces}) of a nonempty finite set $X$ of cardinality $m$ that is closed under containment, that is, if $\sigma \in \mathcal{C}$ and $\alpha \subseteq \sigma$, then $\alpha \in \mathcal{C}$ (we refer the reader to \cite{brownbook} or \cite{C1} for a general combinatorial introduction to complexes). As an example, for any set $X$, the power set of $X$, denoted $\overline{X}$, is a \emph{simplex}, and always a complex. 

We shall need a number of definitions concerning complexes before we return to reliability. Given two complexes $\mathcal{C}_{X}$ and $\mathcal{C}_{Y}$ on disjoint sets $X$ and $Y$, respectively, define the \emph{sum} $\mathcal{C}_{X} \oplus \mathcal{C}_{Y}$ to be the complex on $X \cup Y$ whose faces are those of $\mathcal{C}_{X}$ and $\mathcal{C}_{Y}$ (if $X$ and $Y$ are not disjoint, we take isomorphic disjoint copies to form $\mathcal{C}_{X} \oplus \mathcal{C}_{Y}$). The complex $\mathcal{C}$ is said to be \emph{connected} if it cannot be written as the sum of two other complexes each of positive order, and if $\mathcal{C} = \mathcal{C}_{1} \oplus \mathcal{C}_{2} \cdots \oplus \mathcal{C}_{k}$ where each $\mathcal{C}_{i}$ is connected, then $\mathcal{C}_{1},\ldots,\mathcal{C}_{k}$ are the \emph{components} of $\mathcal{C}$. For a given vertex $x$ of the underlying set, the \emph{deletion} and \emph{link} subcomplexes are those on $X \backslash \{x\}$, with faces
\[ \mbox{del}_{x}\mathcal{C} = \{ \sigma : \sigma \in \mathcal{C},~ x \not\in \sigma\}\]
and
\[ \mbox{link}_{x}\mathcal{C} = \{ \sigma \backslash \{x\} : x \in \sigma \in \mathcal{C} \}. \]

The \emph{$k$-skeleton} of $\mathcal{C}$ is the collection of faces of $\mathcal{C}$ of cardinality $k$. The maximal faces (under containment) are called \emph{facets} or \emph{bases}, and the size of a largest facet is called the \emph{dimension} or \emph{rank} of the complex, and often denoted by $d$ or $r$ (if the complex is empty, that is, has no faces, then we define its dimension to be $-\infty$). A complex is \emph{purely $d$-dimensional} if all of its facets have cardinality $d$.

The \emph{F-vector} of the the complex $\mathcal{C}$ is the sequence $\langle F_{0},F_{1},\ldots,F_{d} \rangle$, where $F_{i}$ is the number of faces of cardinality $i$ in the complex. The \emph{F-polynomial} is the generating function of the $F$-vector, and is given by
\[ f_{\mathcal{C}}(x) = \sum_{i=1}^{d} F_{i}x^{i}\]
(the degree of the polynomial is clearly the dimension of the complex).
The \emph{H-polynomial} of complex $\mathcal{C}$ of dimension $d$ $\mathcal{C}$ is given by 
\[ h_{\mathcal{C}}(x) = (1-x)^{d} f_{\mathcal{C}}\left(\frac{x}{1-x}\right),\]
and the \emph{H-vector} of the complex is the vector of coefficients $\langle H_{0},H_{1},\ldots,H_{d}\rangle$ of the H-polynomial (starting at the constant term); alternatively,
\[ H_{k} = \sum_{i=0}^{k} F_{i}(-1)^{k-i}{{d-i} \choose {k-i}}.\]
It also follows that
\begin{eqnarray}
F_{d} & = & \sum_{i=0}^{d} H_{i}.\label{Fd}
\end{eqnarray}
For a complex $\mathcal{C}$ of dimension $d$ on a set of cardinality $m$,
\begin{eqnarray*}
\Rel{\mathcal{C}} & = & (1-q)^m f_{\mathcal{C}}(q/(1-q))\\
                            & = & (1-q)^{m-d}(1-q)^{d} f_{\mathcal{C}}(q/(1-q))\\
                            & = & (1-q)^{m-d} h_{\mathcal{C}}(q)
\end{eqnarray*}
Thus the roots of $\Rel{\mathcal{C}}$ and $h_{\mathcal{C}}(q)$ coincide, except for some roots at $q = 1$.

\vspace{0.25in}
We are now ready to extend the notion of reliability to complexes.  Given a complex $\mathcal{C}$ on finite set $X$ of cardinality $m$, define the \emph{reliability polynomial} of $\mathcal{C}$ as
\[ \Rel{\mathcal{C}} = \sum_{F \in \mathcal{C}} q^{|F|}(1-q)^{|X\backslash F|}.\]
The idea is that we choose each element of $\mathcal{C}$ independently with probability $q$ and ask the probability that the chosen vertices form a face of the complex. Grouping the terms by their exponent of $q$, we can express the reliability in terms of the F-vector (and F-polynomial) of the complex: 
\[ \Rel{\mathcal{C}} = \sum_{i=0}^{d} F_{i}q^{i}(1-q)^{m-i} = (1-q)^{m}f_{\mathcal{C}}\left( \frac{q}{1-q}\right).\]
It is straightforward to verify that 
\begin{itemize}
\item $\Rel{\mathcal{C}} \not\equiv 0$ if and only if $\mathcal{C}$ has a face (i.e.\ $\mathcal{C}$ is finite dimensional), 
\item $\Rel{\mathcal{C}_{1} \oplus \mathcal{C}_{2}} = \Rel{\mathcal{C}_{1}} \cdot \Rel{\mathcal{C}_{1}}$, and
\item for any element $x \in X$,
\[ \Rel{\mathcal{C}} = (1-q)\cdot \Rel{\mbox{del}_{x}\mathcal{C}} + q\cdot \Rel{\mbox{link}_{x}\mathcal{C}}.\]
\end{itemize}
It follows that if $\mathcal{C}$ has a \emph{loop} $x$ (an element of $X$ that belong to no face) then the removal of $x$ from the underlying set corresponds to division of  $\mathcal{C}$ by $1-q$, and moreover, if $\mathcal{C}$ has a \emph{coloop} $x$ (that is, a vertex in every facet of $\mathcal{C}$), then 
$\mbox{del}_{x}\mathcal{C} = \mbox{link}_{x}\mathcal{C}$, and hence  $\mathcal{C}$ and $\mbox{link}_{x}\mathcal{C}$ have the same reliability. Hence we will often assume that the complex under question has no loops or coloops.\footnote{We remark that in the literature sometimes reliability of complexes has been studied in the guise of \emph{coherent systems} (see, for example, \cite{moore}), which are collections of subsets of a finite set $X$ closed under \emph{superset}. For a coherent system $\mathcal{S}$ on set $X$, its reliability is 
\[ \CRel{\mathcal{S}} = \sum_{S \in \mathcal{S}} p^{|S|}(1-p)^{|X \backslash S|},\]
which is the same as $\Rel{\mathcal{S}}$, where $q = 1-p$ and $\mathcal{S}$ is the complex with faces $X \backslash \sigma$ for $\sigma \in \mathcal{S}$.}

How does this notion of reliability of complexes relate to all-terminal reliability? The set of subsets of the edge set of a connected graph $G$ whose removal leaves the graph connected is indeed a complex, a well-studied one, called the \emph{cographic matroid} of $G$ (a \emph{matroid} is a nonempty complex ${\mathcal M}$ with the \emph{exchange axiom}: if $\sigma,~\alpha \in {\mathcal M}$ and $|\sigma | < |\alpha | $, then there is a $z \in \alpha \backslash \sigma$ with $\{z\} \cup \sigma \in {\mathcal M}$). In fact, the sequence $\langle F_{0},F_{1},\ldots,F_{m-n+1} \rangle$ from the F-form (\ref{Fform}) of the reliability polynomial is in fact the F-vector of the cographic matroid. 
Thus the reliability of the cographic matroid of a graph $G$ is precisely the (all-terminal) reliability of $G$.

\vspace{0.25in}
Let's consider \emph{reliability roots}, that is, the roots of reliability polynomials of nonempty complexes. For a complex of dimension $0$, the reliability is $1$, which obviously has no roots outside the unit disk. For dimension $1$, the reliability is of the form $(1-q)+mq = 1+(m-1)q$, which has all of its roots in the unit disk (and real). The situation changes dramatically when the dimension grows larger.
For example, let $m \geq 2$ and consider the complex $\mathcal{P}_{m}$ on set $X = \{1,2,\ldots,m\}$ with faces 
\[ \emptyset,\{1\},\{2\},\ldots,\{n\},\{1,2\},\{3,4\},\{4,5\},\ldots,\{m-1,m\}.\] 
Clearly the reliability of $\mathcal{P}_{n}$ is 
\[ (1-q)^{2} + mq(1-q) + (m-2)q^{2} = -q^2+(m-2)q+1,\]
which has a root at 
\[ \frac{m-2 + \sqrt{(m-2)^{2}+4}}{2} \sim m-2,\]
which grows arbitrarily large (and positive).
Thus we should insist on some properties (even beyond pureness, as the complex above is pure), shared by cographic matroids, of our complexes if we hope to have roots always in (or close to) the unit disk.

\subsection{Restricting to Shellable Complexes}

A nonempty purely $d$-dimensional complex is \emph{shellable} if its facets can be ordered as $\sigma_{1},\sigma_{2},\ldots,\sigma_{t}$ such that for $j = 2,3,\ldots,t$, the subcomplex
\[ \overline{\sigma_{j}} \cap \bigcup_{i<j} \overline{\sigma_{i}}\]
is a purely $(d-1)$- complex. To illustrate, it is an easy exercise to see that a purely $2$-dimensional complex (i.e. a graph without isolated vertices) is shellable if and only if it is connected as a graph. 
Shellable complexes arise in a variety of combinatorial and topological settings, and have some very nice properties (see, for example, \cite{C1})). 
For a shellable complex, the H-vector is known to consist of nonnegative integers, with a variety of interpretations:
\begin{itemize}
\item An \emph{interval partition} is a collection of disjoint intervals $[L,U] = \{S : L \subseteq S \subseteq U\}$ such that every face in the complex belongs to precisely one interval.  Simplicial complexes that have a partitioning $[L_i, U_i], 1 \leq i \leq b$ with $U_i$ a facet for each $i$ are called \emph{partitionable}.  It is known that all shellable complexes are partitionable \cite[pg. 63 - 64] {C1}, and moreover, $H_{i}$ is the number of lower sets $L_j$ of cardinality $i$.
\item An \emph{order ideal of monomials} $\mathcal{N}$ is a set of monomials closed under division. For every shellable complex, there is an order ideal of monomials $\mathcal{N}$ such that $H_{i}$ counts the number of monomials of degree $i$ in the set.
\end{itemize}
It follows that the sequence $H_{0},H_{1},\ldots, H_{d}$ consists of nonnegative integers with no internal zeros (and in fact, if $\mathcal{C}$ has no coloops, then $H_{d}$ is nonzero).

Many complexes that arise in combinatorial and other settings turn out to be shellable (in particular, matroids are always shellable \cite{C1}). One might hope that the extra condition of shellability on a complex might force the reliability roots inside the unit disk, but such is not the case, even in dimension $2$. In fact, the following is true:

\begin{theorem}
\label{dim2complex}
A purely $2$-dimensional complex $\mathcal{C}$ with $F$-vector $\langle1,m,F_2\rangle$ has a reliability root outside the unit disk if and only if $F_2 \in [\frac{m}{2} ,m -2] \cup [m,2m-5]$.
If, moreover, $\mathcal{C}$ is shellable, then it has a reliability root outside the unit disk if and only if $F_{2} \in [m,2m-5]$.
\end{theorem}
\begin{proof}
Assume to begin with that $\mathcal{C}$ is a purely $2$-dimensional complex; its reliability is given by
\begin{eqnarray*} 
\Rel{\mathcal{C}} & = & F_{2}\cdot q^{2}(1-q)^{m-2} + m\cdot q(1-q)^{m-1} + (1-q)^{m}\\ 
                             & = & (1-q)^{m-2} \left( (F_2 - m + 1) q^{2} + (m-2)q +1\right)
\end{eqnarray*}  

so it suffices to consider the roots of 
\[ r(q)  =  (F_2 - m + 1) q^{2} + (m-2)q +1.\]                     
As the removal of loops cannot introduce any roots (except possibly for $q = 1$), we can assume that  $\mathcal{C}$ has no loops. The complex has no isolated vertices (that is, maximal faces of cardinality $1$), as otherwise the complex would not be pure.  Therefore, we have that $F_2 \geq m/2$. We split the argument up into two cases, depending on whether  $F_{2}$ is less than $m$ or not. 

Note that if $F_2 = m-1$ then $r(q)  = (m-2)q +1$ which has all of its roots in the disk, so we can assume that $F_2 \neq m-1$, and hence we can focus on the roots of 
\[ r_{1}(q)  =  q^{2} + \frac{m-2}{F_2 - m + 1}q +\frac{1}{F_2 - m + 1}.\]  
The Hurwitz Criterion (see \cite{S1}), states that a real polynomial $z^{2}+bz+c$ has all of its roots in the unit disk if and only if $|c| \leq 1$ and $|b| \leq c+1$. 
Thus we set 
\[ b = \frac{m-2}{F_{2}-m+1} \text{ and } c = \frac{1}{F_{2}-m+1}.\]
Clearly the first condition, $|c| \leq 1$, holds as $F_{2}$ is an integer different from $m-1$. So everything hinges on whether
\begin{eqnarray} 
|b| & = & \left| \frac{m-2}{F_{2}-m+1}\right| \leq \frac{1}{F_{2}-m+1} + 1  = c+1. \label{hurwitzineq}
\end{eqnarray}

First consider the case that $F_2 \leq m-1$; as $F_2 \neq m-1$, we have that $F_2 < m - 1$.  Then (\ref{hurwitzineq}) is equivalent to $F_2 \leq 0$, which never holds as dimension $2$ implies $F_2 \geq 1$.  Therefore, there is always a reliability root outside the closed unit disk when $F_2 < m-1$.

Now assume that $F_2 \geq m$.
Then a simple calculation shows that (\ref{hurwitzineq}) holds if and only if $F_{2} \geq 2m-4$. Thus in this case, $\Rel{\mathcal{C}}$ has a root outside the unit disk if and only if $m \leq F_{2} \leq 2m-5$.

Thus we conclude that $\mathcal{C}$ has a reliability root outside the unit disk if and only if $F_2 \in [\frac{m}{2} ,m -2] \cup [m,2m-5]$. A $2$-dimensional complex is shellabe if and only if it is connected as a graph, so shellability implies that $F_2 \geq m-1$, and hence if $\mathcal{C}$ is shellable, then it has a reliability root outside the unit disk if and only if $F_{2} \in [m,2m-5]$.
\end{proof}

Furthermore, when $F_{2} = m \geq 2$, (corresponding to a $2$-dimensional shellable complex whose facets form a unicyclic connected graph), the reliability is
\[ (1-q)^{2} + mq(1-q) + mq^{2} = q^2+(m-2)q+1,\]
which has a root at 
\[ -\frac{m}{2} +1 -\frac{\sqrt{m^2-4}}{2},\] and this root can grow arbitrarily large in absolute value. 

\vspace{0.25in}
Indeed, for other shellable complexes, the roots can be more than unbounded. 

\begin{theorem}
\label{brokencircuitthm}
The reliability roots of shellable complexes are dense in the complex plane.
\end{theorem}
\begin{proof}
The broken circuit complex of a graph $G$ (see \cite{biggs,bjorner}, for example) is formed by fixing a linear order $\preceq$ on the edges of $G$ and declaring any circuit minus its $\preceq$-least edge to be a \emph{broken circuit}.  The broken circuit complex $\mathcal{BR}(G,\preceq)$ is the complex on the edge set of $G$ whose faces are those subsets that do not contain a broken circuit. It is known that every broken circuit complex is shellable \cite{provan}, and that the dimension of the complex, for a graph with $n$ vertices and $c$ components, is $n-c$. 

The interest in broken circuit complexes arises from the surprising fact that if $G$ is a graph of order $n$ and $\langle a_{0},a_{1},\ldots,a_{n-1}\rangle$ is the F-vector of $\mathcal{BR}(G,\preceq)$, then the well-known \emph{chromatic polynomial} of $G$, $\pi(G,x)$, can be expressed as
\[ \pi(G,x) = \sum_{i=0}^{n-c} (-1)^{i}a_{i}x^{n-i}.\]
Sokal \cite{sokalChromPoly} has proven that the roots of chromatic polynomials are dense in the complex plane. Now
\begin{eqnarray*} 
(-q)^n \pi \left( G, \frac{q-1}{q} \right) & = & (-q)^n  \sum_{i=0}^{n-c} (-1)^{i}a_{i}\left( \frac{q-1}{q}\right)^{n-i}\\
                                                   & = & (1-q)^{n} \sum_{i=0}^{n-c} a_{i} \left( \frac{q}{1-q} \right)^{i}\\
                                                   & = & (1-q)^{c}\cdot \Rel{\mathcal{BR}(G,\preceq)}.
 \end{eqnarray*}
As the image (and preimage) of a dense set under a linear fractional transformation is again dense, we see that the roots of the reliability polynomials of broken circuit complexes are also dense in the complex plane. 
\end{proof}

So why is it so hard to find reliability roots of graphs outside the disk? The answer must be that there is a more restrictive property than shellability at play.

\subsection{Matroids}

We have mentioned that every matroid is shellable, but they are indeed a proper subclass, and we have seen that reliability of graphs is the same as the reliability of a certain family of matroids (namely cographic matroids), so it is reasonable to see what happens for the reliability roots of matroids -- do they behave as they do for graphs (and cographic matroids), or does the extension past cographic matroids allow for the kind of extreme behaviour we have seen for other shellable complexes? It is easy to observe that the sum of two matroids is again a matroid, so as reliability is multiplicative over connected components, the reliability roots of a sum of matroids is the union of the reliability roots of the components. Thus we can assume that matroids under consideration for their roots being in the unit disk are connected.

All complexes of dimension at most $1$ are matroids, and we have already seen that such complexes have their reliability roots in the unit disk. In \cite{W1} it was shown that uniform matroids $U_{n,r}$ (those on an set $X$ of size $n$ whose facets are all $r$-subsets of $X$) have their reliability roots in the unit disk. Moreover, it was also shown there that the same is true of reliability roots of cographic matroids of \emph{series-parallel} graphs (those that can be built up from a single vertex by series and parallel operations). As well, we can prove that the real reliability roots of all matroids lie in the unit disk.

\begin{theorem}
The real reliability roots of any matroid $\mathcal{M}$ lie in $[-1,0) \cup \{1\}$, and hence lie in the unit disk.
\end{theorem}
\begin{proof}
We can assume that $\mathcal{M}$ is connected and has no loops or coloops. Note that as
\[ \Rel{\mathcal{C}} = (1-q)^{m-d} \sum_{i=0}^{d}H_{i}q^{i}\]
and all the $H_{i}$ are positive, there are clearly no non-negative roots except $1$.
Moreover, in \cite{BC1} it was shown that for any connected matroid $\M$ of rank $d$ with H-vector $\langle H_{0},H_{1},\ldots,H_{d}\rangle$, any $b \geq 1$ and any $j \in \{0,1,\ldots,d\}$, we have that 
\begin{eqnarray}
(-1)^{j} \sum_{i=0}^{j} (-b)^{i}H_{i} & \geq & 0 \label{BCconditiongen}
\end{eqnarray} 
with equality possible only if $b = 1$. (In particular,
\begin{eqnarray}
\sum_{i=0}^{d} (-1)^{d-i}H_{i} & \geq & 0, \label{BCCondition}
\end{eqnarray}
and we shall often make use of this inequality throughout the paper).
Taking $j = d$ above, we find that 
\begin{eqnarray*}
 (-1)^{d} h_{\mathcal{M}}(-b) & > & 0
 \end{eqnarray*}
for $b >1$, and hence $h_{\mathcal{M}}(q)$ is nonzero for $q \in (1,\infty)$.
It follows that $\Rel{\mathcal{C}}$ has no roots less than $-1$. It follows that all of the real roots of $\Rel{\mathcal{C}}$ lie in $[-1,0) \cup \{1\}$.
\end{proof}

So we see that the usual techniques from {\em real} analysis (such as the Intermediate Value Theorem) that one might use to try to locate roots outside the unit disk are going to fail for matroids, as then any such root must necessarily be nonreal. From \cite{BM1} and \cite{RS1} we know that there are (nonreal) reliability roots of some matroids (namely cographic matroids) that lie a little bit outside the disk. These matroids have dimension $13$ or higher. Might there be matroids of small dimension with reliability roots outside the unit disk?

\section{Reliability Roots of Matroids of Small Dimension}
	\label{sec:Complexes}

Because of earlier remarks, we can assume all matroids (and complexes) have no loops or coloops. From the previous section, any matroid of rank $0$ or $1$, being a complex, has its roots in the unit disk (in fact, any complex of dimension at most $1$ is trivially a matroid). The next case is rank $2$.  We have already seen that for general shellable $2$-dimensional complexes, the reliability roots can be unbounded. However, what about if we insist on the complex being a matroid? Then the situation becomes markedly different.
		
\begin{theorem}
Let $\mathcal{M}$ be a matroid of rank $2$ and order $m$.  Then the roots of $\Rel{\mathcal{M}}$ are all inside the closed unit disk.
\end{theorem}
\begin{proof}
A simple observation is that the graph determined by the $2$-skeleton $G$ of $\mathcal{M}$ of rank $d \geq 2$ must be a complete multipartite graph. (The argument centers on the facts that (a) for $k \leq d$ the $k$-skeleton of a matroid is again a matroid, (b) a graph is complete multipartite if and only if it does not contain $\overline{P_{3}}$, the complement of a path of order $3$, as an induced subgraph that is isomorphic to the complement of a path on $3$ vertices, and (c) a graph is a matroid if and only if it does not contain an induced $\overline{P_{3}}$.) We will show that the number of edges of $G$ (i.e. $F_{2}$ of $\mathcal{M}$) is at least $2m-4$; by Theorem~\ref{dim2complex}, $\mathcal{M}$ will have all of its reliability roots in the unit disk.

Let the cells of the complete multipartite graph $G$ be $V_{1},V_{2},\ldots,V_{k}$. Clearly $k \geq 2$ as $\mathcal{M}$, being of rank $2$, has a face of cardinality $2$. If $k \geq 3$, one can combine cells and decrease the number of edges, so we can assume that $k = 2$. Thus for some $j \in \{2,3,\ldots,\lfloor m/2 \rfloor\}$, $G = K_{j,m-j}$ ($j \neq 1$ as otherwise $\mathcal{M}$ has a coloop, and the proof reverts to the dimension $1$ case). It is easy to see that $G$ has $j(m-j) \geq 2(m-2) = 2m-4$ as the function $g(x) = x(m-x)$ is increasing on $[2,m/2]$.  Thus $G$ has at least $2m-4$ edges, and hence $F_{2} \geq 2m-4$. From Theorem~\ref{dim2complex}, we conclude that all the roots of $\mathcal{M}$ lie in the unit disk.
\end{proof}

We now turn to rank $3$ matroids, where again we can prove that the reliability roots are always in the unit disk (along the way, we prove a new nonlinear inequality among the terms in the $H$-vector of matroids).

\begin{theorem}\label{rank3}
Let $\mathcal{M}$ be a rank $3$ matroid.  Then all the roots of $\Rel{\mathcal{M}}$ lie inside the unit disk.
\end{theorem}
\begin{proof}
Clearly we can assume that $\mathcal{M}$ is connected, as otherwise its reliability is the product of the reliabilities of matroids of smaller rank, and we are done. It can be shown that if a matroid of order $m$ and rank $r$, without loops or coloops, is connected, then $H_{r} \geq m-r$ (see, for example, \cite[pg. 244]{bjorner}), and hence if $m > r$ (which we can always assume, as otherwise the matroid is a simplex with reliability $1$), $H_{r} \geq 1$. We shall make use of this here and throughout this section. 

Farebrother \cite{F1} proved that the roots of a real cubic polynomial $x^3 + a_1x^2 + a_2x + a_3 = 0$ lie in the \emph{open} unit disk $\{z : |z|<1\}$ if and only if the following conditions all hold:
\begin{eqnarray*}
1 + a_1 + a_2 + a_3 & > & 0\\
1 - a_1 + a_2 - a_3 & > &  0\\
3 + a_1 - a_2 - 3a_3 & > & 0\\
1 - a_3^2 + a_1a_3 - a_2 & > &  0
\end{eqnarray*}	
For our purposes, it will be better to rewrite these conditions for the cubic $a_0x^3 + a_1x^2 + a_2x + a_3 = 0$, where all $a_{i}$'s are positive:
\begin{eqnarray*}
a_0 + a_1 + a_2 + a_3 & > & 0\\
a_0 - a_1 + a_2 - a_3 & >  & 0\\
3a_0 + a_1 - a_2 - 3a_3 & > &  0\\
a_0^2 - a_3^2 + a_1a_3 - a_0a_2 & > & 0
\end{eqnarray*}	
It is well known that the roots of a (complex) polynomial depend  continuously on the coefficients \cite{HM1,HM2}. It follows that if we know that if $a_{0},a_{1},a_{2}$ and $a_{3}$ are all positive, then the roots of $a_0x^3 + a_1x^2 + a_2x + a_3 = 0$ are in the \emph{closed} unit disk $\{z : |z|\leq1\}$ if
\begin{eqnarray}
a_0 + a_1 + a_2 + a_3 & \geq & 0 \label{eqn:cubicd1}\\
a_0 - a_1 + a_2 - a_3 &\geq  & 0\label{eqn:cubicd2}\\
3a_0 + a_1 - a_2 - 3a_3 & \geq &  0\label{eqn:cubicd3}\\
a_0^2 - a_3^2 + a_1a_3 - a_0a_2 & \geq & 0\label{eqn:cubicd4}.
\end{eqnarray}		
(If not, there would be a root outside the disk. By increasing $a_{0}$ to $a_{0} + \varepsilon$ and $a_{1}$ to $a_{1}+ \varepsilon/2$, then provided $\varepsilon$ is sufficiently small but positive, we could keep the root outside the disk, but have strict inequality hold in all four conditions, a contradiction to Farebrother's result.)

Consider the reliability polynomial of $\mathcal{M}$:
$$\Rel{\mathcal{M}} = (1-q)^{m-r}(1 + H_1q + H_2q^2 + H_3q^3)$$
Clearly this has $m-r$ roots at $1$, and so we are only interested in the roots of
\begin{eqnarray}
\label{eqn:R3}
h_{\mathcal{M}}(q) & = & 1 + H_1q + H_2q^2 + H_3q^3.
\end{eqnarray}
We set $a_0 = H_3,~a_1 = H_2,~a_2 = H_1~\mbox{and}~a_3 = 1$.
Then rephrasing conditions \reff{eqn:cubicd1}, \reff{eqn:cubicd2}, \reff{eqn:cubicd3} and \reff{eqn:cubicd4}, we need to show
\begin{eqnarray}
H_3 + H_2 + H_1 + 1 & \geq & 0 \label{eqn:cubicd5}\\
H_3 - H_2 + H_1 - 1 &\geq  & 0\label{eqn:cubicd6}\\
3H_3 + H_2 - H_1 - 3 & \geq &  0\label{eqn:cubicd7}\\
H_{3}^{2} - 1 + H_{2} - H_{3}H_{1} & \geq & 0\label{eqn:cubicd8}.
\end{eqnarray}	

Conditions \reff{eqn:cubicd5}, \reff{eqn:cubicd6}, and \reff{eqn:cubicd7} are quite simple to show.  Since all coefficients of the H-vector are nonnegative, \reff{eqn:cubicd5} follows immediately.  For \reff{eqn:cubicd3}, we use \reff{BCCondition}, which implies for $r = 3$ that 
\[ -1 + H_1 - H_2 + H_3 \geq 0,\]
that is, condition \reff{eqn:cubicd6}.
				
For \reff{eqn:cubicd7}, we will use a result of Huy et al \cite{HSZ} that showed that matroids of rank $3$ satisfy Stanley's Conjecture; that is, that their H-vectors are pure $O$-sequences (the H-vectors of pure order ideals of monomials).  In particular, it follows that $H_0 \leq H_1 \leq \dots \leq H_{\lfloor{\frac{r}{2}}\rfloor}$, and $H_i \leq H_{d-i}$ for $ 0 \leq i \leq \lfloor\frac{r}{2}\rfloor$ (see \cite{hibi}).  So $H_3 \geq H_0$ and $H_2 \geq H_1$.  We conclude that
$$3H_3 + H_2 - H_1 - 3 \geq 0,$$
so condition  \reff{eqn:cubicd7} is true.
				
The last inequality to check is \reff{eqn:cubicd8}.  We start by observing that the $2$-skeleton of a rank-$3$ matroid must be a complete $k$-partite graph $G$ with $ k \geq 3$ (as mentioned earlier, for $k \leq r$, the $k$-skeleton of a rank $r$ matroid is also a matroid). Indeed, if it were just a complete bipartite graph, then the dimension of $\mathcal{M}$ would be $2$.  Furthermore, if $k = 3$, then each cell of the complete $k$-partite graph must have cardinality at least $2$, as otherwise there would be a coloop in the matroid.  Therefore, the smallest possible number of edges in $G$ would be a complete $3$-partite graph with independent sets having $2, 2$ and $m-4$ elements, and from this we can determine that $F_2 \geq 4m-12$.  

Using this, the fact that $H_{1} = m-r = m-3$ and by considering the number of faces of cardinality $2$ covered in an interval partition of $\mathcal{M}$, we find that 
\begin{eqnarray}
H_2 =  F_2 - 2H_1 - 3H_0 \geq 2m-9\label{2m-9}
\end{eqnarray}
Finally, from this and \reff{BCCondition} we get
$$H_3 \geq H_2 - H_1 + H_0 \geq m - 5.$$

Now if $H_3 \geq m-3 = H_1$ then \reff{eqn:cubicd8} holds (strictly), as $H_{2} \geq 3$. Moreover, if $H_{3} = m-5$ then 
\begin{eqnarray*}
H_{3}^{2} - 1 + H_{2} - H_{3}H_{1} & = & H_{2}+(-2)(m-5)-1\\
                                                        & = & H_2 -(2m-9)\\
                                                        & \geq & 0
\end{eqnarray*}
by \reff{2m-9}. Our final case is that $H_{3} = m-4$. Here 
\begin{eqnarray*}
H_{3}^{2} - 1 + H_{2} - H_{3}H_{1} & = & H_{2}+(-(m-4))-1\\
                                                        & = & H_2 -(m-3)\\
                                                        & \geq & 0
\end{eqnarray*}                                                    
provided $2m-9 \geq m-3$, that is, provided that $m \geq 6$.
An exhaustive check of all matroids on less than $6$ vertices (whether of rank $3$ or otherwise) can confirm that no root falls outside of the closed unit disk.  Therefore, all matroids of rank $3$ have all of their roots inside of the closed unit disk.
\end{proof}
				
Unfortunately, we were unable to prove that all matroids of rank $4$ have their roots inside the disk, and so we shall focus on paving matroids of rank $4$. A \emph{paving matroid} of rank $r$ on a set $X$ of cardinality $m$ has a complete $(r-1)$-skeleton, that is, all $(r-1)$-subsets of $X$ are faces, so that its H-vector has the form $\langle 1,m-r,\ldots,{{m - r + i -1} \choose i},\ldots,{m-2 \choose r-1}, H_{r}\rangle$, with $H_{r}  = F_{r} - {{m-1} \choose {r-1}}$ \cite[Proposition 3.1]{merino}.  Every uniform matroid is obviously a paving matroid, and it has been conjectured that almost every matroid of order $m$ is a paving matroid.

\begin{theorem} The roots of a paving matroid $\mathcal{M}$ of rank $4$ all fall inside the closed unit disk. \end{theorem}
\begin{proof}
Again, we can assume that $\mathcal{M}$ is connected and has no loops or coloops, and hence $m \geq 5$. If $m = 5$ then the $H$-polynomial of $\mathcal{M}$ has the form $H_{4}q^4 + q^3 + q^2 + q + 1$, while for $m = 6$, it has the form $H_{4}q^4 + 4q^3 + 3q^2 + 2q + 1$, where in either case $H_{4} \geq m-4$. The well-known \emph{Enestr\"om-Kakeya} Theorem (see, for example, \cite{PS1}), which states that if $f(x) = a_0 + a_1x + \dots + a_nx^n$ is a polynomial with $0 < a_1 \leq a_2 \leq \dots \leq a_n$ then the roots of $f(x)$ lie in the (closed) unit disk. Thus the roots of these polynomials are in the unit disk, except possibly for the polynomials $2q^4 + 4q^3 + 3q^2 + 2q + 1$ and $3q^4 + 4q^3 + 3q^2 + 2q + 1$. However, direct calculations in this case show that roots lie in the unit disk, so we can now assume that $m \geq 7$.

For this proof, we will be using Farebrother's \cite{F1} necessary and sufficient conditions for the moduli of all roots of quartic polynomials to fall inside the unit disk: for a real quartic polynomial
$x^4 + a_1x^3 + a_2x^2 + a_3x + a_4 = 0$, the roots fall inside the open unit disk if and only if 
\begin{eqnarray*}
1 &> & a_4 \\
3 + 3a_4 &> & a_2\\
1 + a_1 + a_2 + a_3 + a_4 &> &  0\\
1 - a_1 + a_2 - a_3 + a_4 &> & 0 \\
(1-a_4)(1-a_4^2) - a_2(1-a_4)^2 + (a_1 - a_3)(a_3 - a_1a_4) &> &0.
\end{eqnarray*}		
By a similar argument as we used in Theorem \ref{rank3} regarding adjusting coefficients slightly to force roots to be inside of the disk in the case where there is an equality, it follows that the roots are in the (closed) unit disk if	
\begin{eqnarray}
1 &> & a_4 \label{eqn:quarticd1}\\
3 + 3a_4 & \geq & a_2 \label{eqn:quarticd2}\\
1 + a_1 + a_2 + a_3 + a_4 & \geq &  0\label{eqn:quarticd3}\\
1 - a_1 + a_2 - a_3 + a_4 & \geq & 0 \label{eqn:quarticd4}\\
(1-a_4)(1-a_4^2) - a_2(1-a_4)^2 + (a_1 - a_3)(a_3 - a_1a_4) &> &0.\label{eqn:quarticd5}
\end{eqnarray}	

The H-polynomial of $\mathcal{M}$ (whose roots we are interested in) is given by 
\begin{eqnarray*}
h_{\mathcal{M}}(q) & = & H_0 + H_1 q + H_2 q^2 + H_3 q^3 + H_4 q^4\\
                   & = & 1 + (m-4)q+{{m-3} \choose 2}q^2 + {{m-2} \choose 3}q^3 + H_4q^4\\
			    & = & H_4 \left( \frac{1}{H_4} + \frac{m-4}{H_4}q + \frac{{{m - 3} \choose 2}}{H_4}q^2 + \frac{{{m-2} \choose 3}}{H_4}q^3 + q^4 \right) .
\end{eqnarray*}
By \reff{eqn:quarticd1}--\reff{eqn:quarticd5}, we need only show that
\begin{eqnarray}
H_4 & > & 1 \label{eqn:quarticd6}\\
3H_4 + 3 & \geq & H_2 \label{eqn:quarticd7}\\
H_4 + H_3 + H_2 + H_1 + 1 & \geq &  0\label{eqn:quarticd8}\\
H_4 - H_3 + H_2 - H_1+ 1 & \geq & 0 \label{eqn:quarticd9}\\
(H_4 - 1)(H_{4}^{2} - 1) - H_2(H_4 - 1)^2 + (H_3 - H_1)(H_4H_1-H_3) & > &0.\label{eqn:quarticd10}
\end{eqnarray}	

Clearly condition \reff{eqn:quarticd6} holds as $H_4 \geq m-4 \geq 3$. Condition \reff{eqn:quarticd8} holds since all of the $H_i$ are positive. Condition \reff{eqn:quarticd9} follows directly from  inequality \reff{BCCondition}, and implies 
\begin{eqnarray}
	H_4 &\geq & H_3 - H_2 + H_1 - H_0 = \frac{1}{6}m^3 - 2m^2 + \frac{53}{6}m - 15.	\label{eqn:H4BCineq}
\end{eqnarray}
It follows that 
\begin{eqnarray*}
3H_4 + 3 \geq \frac{1}{2}m^3-6m^2+\frac{53}{2}m-42 > \frac{1}{2}m^2-\frac{7}{2}m+6 =  H_2
\end{eqnarray*}
as $m \geq 7$, so condition \reff{eqn:quarticd7} holds. 

All that remains is condition \reff{eqn:quarticd10}. We set $l = H_1 = m-4$ (which is at least $3$). We substitute into the left-hand side of \reff{eqn:quarticd10} the values for $H_{i} = {{l+i-1} \choose i}$ for $i \leq 3$ and set $z = H_{4}$. We need to show that 
\begin{eqnarray}
(z - 1)(z^2 -1) - {{l+1} \choose 2}(z-1)^{2} + & & \nonumber \\
\left( {{l+2} \choose 3} - l \right) \left( zl - {{l+2} \choose 3} \right) & > & 0.\label{eqn:P4ineq5inz}
\end{eqnarray}

Let's denote the left-hand side of \reff{eqn:P4ineq5inz} by $f = f(z)$.
The derivative $f^{\prime}(z)$ is a quadratic in $z$ that opens up, and one can verify that its discriminant is $-l^4-4l^3+l^2-8l+16$, which is positive for $l \geq 2$. Thus $f$ is increasing on $[2,\infty)$, so that as $l \geq 3$, $z \geq \frac{1}{6}m^3 - 2m^2 + \frac{53}{6}m - 1 = \frac{1}{6}l^3 + \frac{5}{6}l - 1 > 2$
\begin{eqnarray*} 
f(z) & \geq & f\left( \frac{1}{6}l^3 + \frac{5}{6}l - 1 \right) \\
     & = & \frac{l(l-2)(l-1)(l^2-l+12)(l^4+l^3+5l^2-l+12)}{216}\\
     & > & 0
\end{eqnarray*}
(both the polynomials $l^2-l+12$ and $l^4+l^3+5l^2-l+12$ have only complex roots, and hence are always positive). Therefore, \reff{eqn:quarticd10} holds, and we are done.
\end{proof}
				
\section{Reliability Roots of Matroids Outside the Disk}
		In the previous section, we showed that the roots of matroids of rank at most $3$ and paving matroids of rank $4$ have their roots inside of the unit disk. There are, to be sure, reliability roots of matroids outside the disk --  we know from network reliability that there are graphs whose all terminal reliability roots have modulus greater than $1$, and hence there are cographic matroids that have roots outside of the disk. Are there other matroids with reliability roots outside the unit disk?  		
		There exist, of course, operations on matroids that would yield roots outside of the unit disk. The direct sum $\M_1 \oplus \M_2$ of matroids $\M_1$ and $\M_2$, on disjoint sets $X_{1}$ and $X_{2}$, respectively, has as its faces $\sigma_{1} \cup \sigma_{2}$, where $\sigma_i$ is a face of $\M_1$ ($i = 1,~2$). It is easy to see that $\Rel{\M_1 \oplus \M_2} = \Rel{\M_1}\Rel{\M_2}$, and so the reliability roots of $\M_1 \oplus \M_2$ is simply the union of the reliability roots of $\M_1$ and $\M_2$.  Therefore, if either $\M_1$ or $\M_2$ have roots outside of the disk, then $\M_1 \oplus \M_2$ will have roots outside of the disk.  It follows that we can embed any matroid in another that has a root outside of the unit disk. However, this seems somewhat artificial. 
        
        Two other operations, though, yield other matroids with reliability roots outside the disk. We define a \emph{$k$-thickening} of a matroid at vertex $v$, denoted by $\mbox{Th}(\mathcal{M},v,k)$, to be the matroid such that 
\begin{itemize}
	\item if $v$ is an element in a face $\sigma \in \mathcal{M}$ then 
$$(\sigma \setminus \{v\}) \cup \{v_1\}, (\sigma \setminus \{v\}) \cup \{v_2\}, \dots , (\sigma \setminus \{v\}) \cup \{v_k\} \in \mbox{Th}(\mathcal{M},v,k),$$
\item if $v$ is not an element in a face $\sigma \in \M$, then $\sigma \in \mbox{Th}(\M,v,k)$. 
\end{itemize}
(In essence, we place $k-1$ new elements in parallel to $v$.)
A calculation shows that the H-polynomial of $\mbox{Th}(\mathcal{M},v,k)$ is given by
		$$h_{\mbox{Th}(\mathcal{M},v,k)}(q) = h_{\del_\M(v)}(q) + kxh_{\lnk_\M(v)}(q).$$
As $k$ grows large $kxh_{\lnk_\M(v)}(q)$ will dominate over $h_{\del_\M(v)}(q)$, that is, if we consider $(1/k)h_{\mbox{Th}(\mathcal{M},v,k)}(q)$, which has the same roots as $h_{\mbox{Th}(\mathcal{M},v,k)}(q)$, it approaches coefficient-wise to $xh_{\lnk_\M(v)}(q)$  which leads to the roots of $h_{\mbox{Th}(\mathcal{M},v,k)}(q)$ being close to those of $h_{\lnk_\M(v)}(q)$ (and $0$).  Therefore, if $h_{\lnk_\M(v)}(q)$ has a root outside of the disk, then $h_{\mbox{Th}(\mathcal{M},v,k)}(q)$ will also have a root outside of the disk provided $k$ is large enough.
		
Another operation that we will focus on is a generalization of replacing an edge of a graph with $k$ edges in parallel (an operation which Royal and Sokal used on a pair of opposite edges to construct their graph).	
We define the \emph{$k$-replacement} at vertex $v$, denoted by $\mbox{Rep}(\M,v,k)$, to be the matroid such that
		\begin{itemize}
			\item if $v$ is an element of a face $\sigma \in \mathcal{M}$ then $(\sigma \setminus \{v\}) \cup \{v_1, v_2, \dots, v_k\} \in $ $\mbox{Rep}(\M,v,k)$; and
			\item if $v$ is not an element of a face $\sigma \in \mathcal{M}$ then for any $\tau \subsetneq \{v_1, v_2, \dots, v_k\}, \sigma \cup \tau \in \mbox{Rep}(\M,v,k)$.
		\end{itemize}
		If we use a $k$-replacement on a single element $v$, then we get the $f$-polynomial
		\begin{equation}
			\label{eqn:genrelFPoly} 
			f_{\mbox{Rep}(\M,v,k)}(q) =q^kf_{\lnk_\M(v)}+ ((1 + q)^k-q^k)f_{\del_\M(v)}
		\end{equation}
However, it would be more beneficial to be able to use a $k$-replacement on multiple vertices simultaneously, with different values of $k$, and so we will introduce a multivariate generating polynomial of the matroid.  Given a matroid $\M$ on elements $v_1, v_2, \dots v_m$, for each face $\sigma \in \M$ we introduce the variable $q_i$ if $v_i \in \sigma$ and $p_i$ if $v_i \not\in\sigma$, and define the \emph{multivariate generating polynomial} by
\begin{eqnarray*}
\mgen{q_{1},p_{1},\ldots,q_{m},p_{m}} & = & \sum_{\sigma \in \M} \prod_{v_i \in \sigma} q_{i} \prod_{v_j \not\in \sigma} p_{j}.
\end{eqnarray*}
For example, consider the matroid $\M$ with facets
		$$\M = \{ \{1,2,3\},\{1,2,4\},\{2,3,4\}\}.$$
The faces of $\M$ are therefore 
\begin{eqnarray*}
& &\{ \emptyset,\{1\},\{2\},\{3\},\{4\},\{1,2\},\{1,3\},\{1,4\},\{2,3\},\{2,4\},\{3,4\},\{1,2,4\},\\
& &\{2,3,4\},\{1,2,3\}\}.
\end{eqnarray*}
The multivariate generating polynomial for $\M$ is given by 
\begin{eqnarray*}
& & p_1p_2p_3p_4+p_1p_2p_3q_4+p_1p_2p_4q_3+p_1p_2q_3q_4+p_1p_3p_4q_2+p_1p_3q_2q_4+\\
& & p_1p_4q_2q_3+p_1q_2q_3q_4+p_2p_3p_4q_1+p_2p_3q_1q_4+p_2p_4q_1q_3+p_3p_4q_1q_2+\\
& & p_3q_1q_2q_4+p_4q_1q_2q_3.
\end{eqnarray*}
For $\mathbf{k}=(k_1,k_2,\ldots,k_m)$, the reliability polynomial of the $\mathbf{k}$-replacement on all vertices, obtained by sequentially carrying out a $k_i$-replacement at vertex $v_i$, has reliability given by 
		\begin{equation}
			\label{eqn:genrelPR1}
\mgen{q^{k_1},1-q^{k_1}, q^{k_2}, 1-q^{k_2}, \dots, q^{k_m}, 1 -q^{k_m}}.
		\end{equation}

		

As an example, suppose we start with Cog$(K_4)$, and carry out two $6$-replacements of a pair of parallel edges in the $K_{4}$ (which are vertices in its cographic matroid) to get matroid $\mathcal{RS}$. Then 
\begin{eqnarray*}
			\label{eqn:CogRS}
			\Rel{\text{Cog}(RS)} & = & \mgen{q^6, 1-q^6, q, 1-q, q^6, 1-q^6, q, 1-q, q, 1-q, q, 1-q}\\            
             & = & (1-q)^3 \left( 6q^{13}+10q^{12}+14q^{11}+18q^{10}+22q^9+26q^8+ \right.\\
			& & \left. 26q^7+22q^6+18q^5+14q^4+10q^3+6q^2+3q+1 \right)
\end{eqnarray*}
which is indeed the reliability of the Royle-Sokal graph, and has a root outside the disk.

		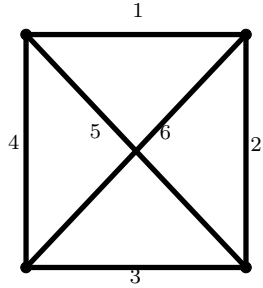
\begin{figure}[h!]
			\begin{center}
			\begin{tikzpicture}[line cap=round,line join=round,>=triangle 45,x=1.4617866644141715cm,y=1.5495027694673518cm]
\clip(0.5920883976871424,0.8502082193129823) rectangle (5.380748857115628,4.7224181425405);
\draw [line width=2.pt] (2.,4.)-- (2.,2.);
\draw [line width=2.pt] (4.,2.)-- (2.,2.);
\draw [line width=2.pt] (4.,4.)-- (4.,2.);
\draw [line width=2.pt] (4.,4.)-- (2.,4.);
\draw [line width=2.pt] (2.,4.)-- (4.,2.);
\draw [line width=2.pt] (2.,2.)-- (4.,4.);
\begin{scriptsize}
\draw [fill=black] (2.,4.) circle (2.0pt);
\draw [fill=black] (2.,2.) circle (2.0pt);
\draw[color=black] (1.888329246187612,3.075284070719839) node {$4$};
\draw [fill=black] (4.,2.) circle (2.0pt);
\draw[color=black] (2.9946749385383313,1.9276568226843918) node {$3$};
\draw [fill=black] (4.,4.) circle (2.0pt);
\draw[color=black] (4.0927643197521055,3.0587714484459476) node {$2$};
\draw[color=black] (3.0194438719491683,4.206398696481395) node {$1$};
\draw[color=black] (2.6313972485127217,3.174359804363187) node {$5$};
\draw[color=black] (3.267133206057538,3.1578471820892955) node {$6$};
\end{scriptsize}
\end{tikzpicture}
\end{center}
\label{fig:K4}
\caption{$K_4$ with edges labeled $\{1,2,\dots,6\}$}
\end{figure}
		
The upshot is that we can start with any matroid and carry out a $\mathbf{k}$-replacement in the hope of finding other reliability roots outside the unit disk. Using this approach we can indeed find a connected matroid that is not a cographic matroid with a reliability root outside the disk, as follows.
Consider the well-known \emph{Fano Matroid}, $F_7$, which is the matroid of order $7$ shown in Figure \ref{Fig:F7} (its facets are all $3$-tuples of the set $\{1,2,3,4,5,6,7\}$ except for the $3$-tuples whose vertices form a line or circle in the diagram).  This matroid is connected and is known to be non-cographic (and non-graphic) \cite[pg. 643-644]{O1}.

			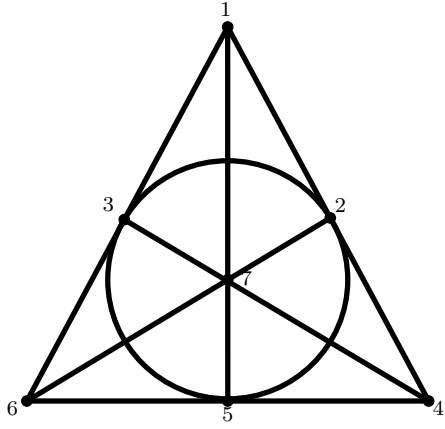
\begin{figure}[h]   
				\begin{center}
				\begin{tikzpicture}[line cap=round,line join=round,>=triangle 45,x=2.9556153402699046cm,y=2.9347665626550787cm]
\clip(1.757916463839754,2.6222358917995665) rectangle (4.314656481695169,4.689668768306799);
\draw [line width=2.pt] (3.,4.5)-- (2.096944598379837,2.804482965175766);
\draw [line width=2.pt] (3.,4.5)-- (3.904066075689833,2.804482965175766);
\draw [line width=2.pt] (3.904066075689833,2.804482965175766)-- (2.096944598379837,2.804482965175766);
\draw [line width=2.pt] (3.,4.5)-- (3.000560050447094,2.804482965175766);
\draw [line width=2.pt] (2.096944598379837,2.804482965175766)-- (3.461511828457688,3.634464628234374);
\draw [line width=2.pt] (3.904066075689833,2.804482965175766)-- (2.534897765352633,3.626754836564147);
\draw [line width=2.pt] (3.0003785427883396,3.353985834009123) ellipse (1.5952446400212388cm and 1.5839918561126236cm);
\begin{scriptsize}
\draw [fill=black] (3.,4.5) circle (2.0pt);
\draw[color=black] (2.9922047483216785,4.581668543414629) node {$1$};
\draw [fill=black] (2.096944598379837,2.804482965175766) circle (2.0pt);
\draw[color=black] (2.031223155403609,2.76990966869294) node {$6$};
\draw [fill=black] (3.904066075689833,2.804482965175766) circle (2.0pt);
\draw[color=black] (3.94877816879517,2.76990966869294) node {$4$};
\draw [fill=black] (3.000560050447094,2.804482965175766) circle (2.0pt);
\draw[color=black] (3.001021093210835,2.74346063402547) node {$5$};
\draw [fill=black] (3.461511828457688,3.634464628234374) circle (2.0pt);
\draw[color=black] (3.50796092433734,3.691217709609809) node {$2$};
\draw [fill=black] (2.534897765352633,3.626754836564147) circle (2.0pt);
\draw[color=black] (2.463224054972282,3.6956258820543875) node {$3$};
\draw [fill=black] (3.0003785427883396,3.353985834009123) circle (2.0pt);
\draw[color=black] (3.0847763696578228,3.360604776266435) node {$7$};
\end{scriptsize}
\end{tikzpicture}
\caption{Fano Plane of order $7$}                    	
				\label{Fig:F7}
\end{center}
\end{figure}
Its multivariate generating polynomial is given by
\begin{align*}
	\text{gren}(F_7)=&p_1p_2p_3p_4p_5p_6p_7+p_1p_2p_3p_4p_5p_6q_7+p_1p_2p_3p_4p_5p_7q_6+p_1p_2p_3p_4p_5q_6q_7+\\
	&p_1p_2p_3p_4p_6p_7q_5+p_1p_2p_3p_4p_6q_5q_7+p_1p_2p_3p_4p_7q_5q_6+p_1p_2p_3p_4q_5q_6q_7+\\
	&p_1p_2p_3p_5p_6p_7q_4+p_1p_2p_3p_5p_6q_4q_7+p_1p_2p_3p_5p_7q_4q_6+p_1p_2p_3p_5q_4q_6q_7+\\
	&p_1p_2p_3p_6p_7q_4q_5+p_1p_2p_3p_6q_4q_5q_7+p_1p_2p_4p_5p_6p_7q_3+p_1p_2p_4p_5p_6q_3q_7+\\
	&p_1p_2p_4p_5p_7q_3q_6+p_1p_2p_4p_5q_3q_6q_7+p_1p_2p_4p_6p_7q_3q_5+p_1p_2p_4p_6q_3q_5q_7+\\
	&p_1p_2p_4p_7q_3q_5q_6+p_1p_2p_5p_6p_7q_3q_4+p_1p_2p_5p_7q_3q_4q_6+p_1p_2p_6p_7q_3q_4q_5+\\
	&p_1p_3p_4p_5p_6p_7q_2+p_1p_3p_4p_5p_6q_2q_7+p_1p_3p_4p_5p_7q_2q_6+p_1p_3p_4p_6p_7q_2q_5+\\
	&p_1p_3p_4p_6q_2q_5q_7+p_1p_3p_4p_7q_2q_5q_6+p_1p_3p_5p_6p_7q_2q_4+p_1p_3p_5p_6q_2q_4q_7+\\
	&p_1p_3p_5p_7q_2q_4q_6+p_1p_3p_6p_7q_2q_4q_5+p_1p_4p_5p_6p_7q_2q_3+p_1p_4p_5p_6q_2q_3q_7+\\
	&p_1p_4p_5p_7q_2q_3q_6+p_1p_5p_6p_7q_2q_3q_4+p_2p_3p_4p_5p_6p_7q_1+p_2p_3p_4p_5p_6q_1q_7+\\
	&p_2p_3p_4p_5p_7q_1q_6+p_2p_3p_4p_5q_1q_6q_7+p_2p_3p_4p_6p_7q_1q_5+p_2p_3p_4p_7q_1q_5q_6+\\
	&p_2p_3p_5p_6p_7q_1q_4+p_2p_3p_5p_6q_1q_4q_7+p_2p_3p_5p_7q_1q_4q_6+p_2p_3p_6p_7q_1q_4q_5+\\
	&p_2p_4p_5p_6p_7q_1q_3+p_2p_4p_5p_6q_1q_3q_7+p_2p_4p_6p_7q_1q_3q_5+p_2p_5p_6p_7q_1q_3q_4+\\
	&p_3p_4p_5p_6p_7q_1q_2+p_3p_4p_5p_6q_1q_2q_7+p_3p_4p_5p_7q_1q_2q_6+p_3p_4p_6p_7q_1q_2q_5+\\ &p_4p_5p_6p_7q_1q_2q_3
\end{align*}
There are, we have found, many choices for $\mathbf{k}$-replacements that result in roots outside of the unit disk.  For example, consider $\mathbf{k} = (1,4,4,4,5,4,5)$.  Then we obtain that the reliability of the matroid formed is 
\begin{eqnarray*}
& & (1-q)^4(13q^{23}+46q^{22}+106q^{21}+200q^{20}+311q^{19}+418q^{18}+\\
& & 504q^{17}+552q^{16}+557q^{15}+528q^{14}+476q^{13}+411q^{12}+\\
& & 343q^{11}+278q^{10}+218q^9+165q^8+120q^7+84q^6+\\
& & 56q^5+35q^4+20q^3+10q^2+4q+1)
\end{eqnarray*}
which has maximum modulus approximately $1.00184754528486$. 
We can iterate through all possible combinations of the $k_i$ in the proof above between $1$ and $5$, which yields many roots that are outside of the unit disk.  We have included some notable roots in the table below:
\begin{center}
\begin{tabular}{| c | c |}
 				\hline			
  				\textbf{k} & Maximum Modulus \\ \hline \hline
				$\{1, 4, 4, 4, 5, 4, 5\}$ & $1.0018475452848614$ \\
     $   \{2, 2, 5, 2, 5, 5, 5\} $&$1.003722670361891$ \\
        $\{3, 3, 3, 5, 5, 5, 3\}$ & $1.001595847748084$ \\
       $\{3, 3, 5, 3, 5, 5, 5\} $& $1.0070841870536522$ \\
       $\{4, 4, 4, 5, 5, 5, 4\} $& $1.0076584896344196$ \\
       $\{4, 4, 5, 4, 5, 5, 5\} $& $1.0087285165185493$ \\\hline

  			\end{tabular}
\end{center}

		
		\section{Open Problems}
There are still many questions open on reliability roots for complexes in general, and matroids in particular. We have seen that the reliability roots of all matroids of small rank (rank at most $3$) are in the unit disk.

\begin{prob}
What is the smallest rank (or order) of a matroid with a reliability root outside the unit disk?
\end{prob}

The cographic matroids corresponding to the Royle-Sokal graph has rank $13$ and order $16$, and is the smallest one we know.


It would be of interest to find other constructions (other than those raised in section 4) that produce reliability roots outside the unit disk, both for matroids and other complexes.
The most salient open question is how large in moduli can a reliability root of a matroid be?

\begin{prob}
Are the reliability roots of matroids bounded?
\end{prob}

It seems likely that they are, perhaps even by $2$, but of course there may be some extremal families that have roots far outside the disk.


We have seen that the paving matroids of rank $4$ (and smaller rank) have roots inside the unit disk. Almost all of the coefficients of the H-polynomial of paving matroids are completely described -- only the leading coefficient varies from one paving matroid of order $m$ and rank $r$.

\begin{prob}
Are the reliability roots of paving matroids always in the unit disk?
\end{prob}

Paving matroids are widely believed to dominate all matroids -- it has been conjectured that almost all matroids of order $m$ are paving matroids. We have found throughout our work that the reliability roots of matroids are rarely outside the disk:

\begin{prob}
Do almost all matroids have their reliability roots inside the unit disk?
\end{prob}

The difficulty on approaching this question is that, unlike graphs (which have the Erd\"{o}s--R\'{e}nyi model), there is no generative probabilistic model for matroids (it is this issue that has hampered attempts to show that almost all matroids are paving). 

To conclude, we will show that while we do not know whether almost all matroids have their reliability roots in the unit disk, for almost all purely $d$-dimensional complexes, the reliability roots are all in the unit disk. As the purely $d$-dimensional complexes on $[m] = \{1,2,\ldots,m\}$ are in a $1-1$ correspondence with the collections of $d$-subsets of $[m]$, we can form a generative probabilistic model $\mathcal{PD}_{n,1/2}$ for purely $d$-dimensional complexes on $[m]$ by randomly choosing facets, that is, each $d$-subset of $[m]$ independently with probability $1/2$ (under this model, each purely $d$-dimensional complex on $[n]$ occurs with equal probability); we extend the model (as is done for graphs) to $\mathcal{PD}_{m,p}$ by fixing any $p \in (0,1)$ and choosing each $d$-subset of $[m]$ independently with probability $p$. 

\begin{theorem}
For fixed $d$ and $p \in (0,1)$, the reliability roots of almost all purely $d$-dimensional complexes in $\mathcal{PD}_{m,p}$ lie inside the unit disk.
\end{theorem}
\begin{proof}
Clearly we can assume $d \geq 2$ (as all complexes of dimension at most $1$ have their reliability roots in the unit disk). Let $\varepsilon > 0 $. Consider the following events:
\begin{itemize}
\item $E_{1}$ is the event that there are no loops.
\item $E_{2}$ is the event that every $d-1$ subset of $[m]$ is a subset of a facet.
\item $E_{3}$ is the event that the number of facets is greater than $(1-\varepsilon)p{m \choose d}$.
\end{itemize}
Now 
\begin{eqnarray}
\mbox{Prob}(\overline{E_{1}}) & \leq & m(1-p)^{{n-1} \choose {d-1}} = o(1), \mbox{ and }\label{prob1}\\
\mbox{Prob}(\overline{E_{2}}) & \leq & {m \choose {d-1}}(1-p)^{m-(d-1)} = o(1).\label{prob2}
\end{eqnarray}
Moreover, the number of facets has a binomial distribution with probability $p$ and $M = {m \choose d}$ (and hence mean $pM$). The well-known Chernoff lower tail bounds \cite{chernoff} implies that for independent random variables $X_{1},\ldots,X_{M}$, with each $X_{i}$ always lying in $[0,1]$, if we set $X = \sum X_{i}$ and $\mu = E(X)$, then 
\[ \mbox{Prob}(X \leq (1-\varepsilon)\mu) \leq e^{-\mu \varepsilon^2 /2}.\]
It immediately follows that
\begin{eqnarray}
\mbox{Prob}(\overline{E_{3}}) & \leq & e^{-pM \varepsilon^2 /2} = o(1).\label{prob3}
\end{eqnarray}

From \reff{prob1}, \reff{prob2} and \reff{prob3} we see that 
\[ \mbox{Prob}(\overline{E_{1} \cap E_{2} \cap E_{3}}) = \mbox{Prob}(\overline{E_{1}} \cup \overline{E_{2}} \cup \overline{E_{3}})  = o(1),\]
so that 
\[ \lim_{m \rightarrow \infty} \mbox{Prob}(E_{1} \cap E_{2} \cap E_{3}) = 1.\]
Consider any purely $d$-dimensional complex $\mathcal{C}$ that lies in $E_{1} \cap E_{2} \cap E_{3}$. From $E_{1}$ and $E_{2}$, we see that the $k$-skeletons are full for all $k< d$, so that $F_{i} = {m \choose i}$ for $i < d$, and $E_{3}$ implies that $F_{d} > (1-\varepsilon)p{m \choose d}$. 


It is not hard to verify (see \cite[Proposition 6.3]{W1}) that the H-vector of the uniform matroid $U(m,d)$ is $\langle {{m-d-1} \choose {0}}, {{m-d} \choose {1}},\ldots, {{m-d+i-1} \choose {i}},\ldots, {{m-1} \choose {d}}\rangle$. The uniform matroid $U(m,d)$ and pure complex $\mathcal{C}$ have the same F-vector vector except for possibly $F_{d}$, and $H_i$ of an H-vector only depends on $F_{j}$ for $j \leq i$, and so we conclude that for $\mathcal{C}$, 
\[ H_{i} = {{m-d+i-1} \choose {i}}\]
for $i = 0,1,\ldots,d-1$.
Now from a simple binomial identity, we find that 
\[ \sum_{i=0}^{d-1} {{m-d+i-1} \choose {i}} = {{m-1} \choose {d-1}},\]
and so for $\mathcal{C}$, from \reff{Fd} we find that 
\[H_{d} = F_{d} - \sum_{i=0}^{d-1} {{m-d+i-1} \choose {i}} > (1-\varepsilon)p{m \choose d} - {{m-1} \choose {d-1}}.\]

It is trivial to check that 
\[ H_{i} = {{m-d+i-1} \choose {i}} \leq H_{i+1} = {{m-d+i} \choose {i+1}}\]
provided $m \geq d+1$ (which we can assume, as we are interested in the limit as $m \rightarrow \infty$). Moreover, 
\[ H_{d} = {m \choose d} - {{m-1} \choose {d-1}} \geq H_{d-1} = {{m-2} \choose {d-1}}\]
provided 
\[ (1-\varepsilon)pm(m-1) \geq d(m+d) + (m-1)d,\]
which clearly holds if $m$ is sufficiently large (as $d$ and $p$ are fixed).
Thus the (positive) coefficients of the $h$-polynomial of $\mathcal{C}$ are nondecreasing, so we conclude by the Enestr\"om-Kakeya Theorem mentioned earlier that all of roots of the polynomial have modulus at most $1$, and hence the same is true of the roots of the reliability polynomial of $\mathcal{C}$. It follows that for almost all complexes in $\mathcal{PD}_{m,p}$, their reliability roots lie inside the unit disk.
\end{proof}

Are almost all purely $d$-dimensional complexes on $[m]$ shellable? It is true for $d = 2$ (as almost all graphs are connected), but it seems unlikely if $d \geq 3$. But if so, then almost all shellable complexes would have their roots in the unit disk.  As well, while every matroid is a pure complex, we do not know whether the H-vector of almost all matroids is nondecreasing (which would be sufficient to proving that the roots of almost matroids are in the unit disk).
	
\vspace{0.25in}
\noindent {\bf Acknowledgements}
\vspace{0.1in}

\noindent J.I. Brown acknowledges support from NSERC (grant application RGPIN 170450-2013).

\end{document}